\documentclass{article}

\usepackage{etoolbox} 
\usepackage[shortlabels,inline]{enumitem}

\usepackage[immediate]{silence}
\usepackage{environ}

\usepackage[
	a4paper,
	margin=25mm
]{geometry}
\usepackage[onehalfspacing]{setspace}

\usepackage{sectsty} 
\DeactivateWarningFilters[temp] 

\usepackage{amsmath, amssymb, amsthm}
\usepackage{mathtools}
\usepackage[a]{esvect}
\usepackage{bbm}

\usepackage{graphicx}
\usepackage[dvipsnames]{xcolor}
\usepackage{tikz}
\usepackage[pstarrows]{pict2e}

\usepackage[
	setpagesize=false,
	colorlinks=true,
	linkcolor=blue,
	citecolor=blue,
	pdfencoding=auto,
	psdextra,
]{hyperref}
\usepackage{cleveref}
\usepackage[
    sorting=nyt,
    maxbibnames=99
]{biblatex}

\setlist[enumerate,1]{label=(\arabic*), ref=(\arabic*)}
\setlist[enumerate,3]{label=(\roman*), ref=(\roman*)}

\allsectionsfont{\boldmath} 

\theoremstyle{plain}
\newtheorem{theorem}{Theorem}[section]
\newtheorem{lemma}[theorem]{Lemma}

\newtheorem{observation}[theorem]{Observation}

\newtheorem{claim}{Claim}[theorem]
\newtheorem*{claim*}{Claim}
\newtheorem{rmk}{Remark}

\makeatletter
\newenvironment{claimproof}[1][Proof]{\par
	\pushQED{\qed}%
	
	\normalfont \topsep6\p@\@plus6\p@\relax
	\trivlist
	\item[\hskip\labelsep
	\textit{#1}\@addpunct{.}~]\ignorespaces
}{%
	\popQED\endtrivlist\@endpefalse
}
\makeatother

\newlist{Cases}{enumerate}{3}
\setlist[Cases]{parsep=0pt plus 1pt}
\setlist[Cases,1]{wide=0pt, listparindent=\parindent,
    label = \textbf{Case~\arabic*:}}
\setlist[Cases,2]{wide=\parindent, listparindent=\parindent,
    label = \textbf{Case~\arabic{Casesi}-\arabic{Casesii}:}}

\newcounter{case}
\AtBeginEnvironment{proof}{\setcounter{case}{0}}

\crefname{case}{case}{cases}

\theoremstyle{definition}
\newtheorem{definition}[theorem]{Definition}

\newcounter{statement}
\makeatletter
\let\c@statement\c@equation
\makeatother

\NewEnviron{statement}
	{\vspace{1ex}\par\noindent
	\refstepcounter{statement}%
	\hspace{1.5cm}%
	\parbox{\dimexpr\textwidth-3cm}{\BODY}%
	~\hspace{-1.55cm}\llap{(\arabic{statement})}%
	\vspace{1ex}\par}

\Crefformat{statement}{(#2#1#3)}

\newcommand{\calC}{\mathcal{C}}

\newcommand{\calG}{\mathcal{G}}
\newcommand{\calH}{\mathcal{H}}

\newcommand{\calS}{\mathcal{S}}

\newcommand{\calW}{\mathcal{W}}
\newcommand{\capa}{{\rm cap}}

\newcommand{\ld}{\ell_{\delta}}
\newcommand{\td}{\tau_{\delta}}
\newcommand{\nw}{N_{\calW}}

\newcommand{\ve}{\varepsilon}

\newcommand{\defeq}{\coloneqq}

\let\originalleft\left
\let\originalright\right
\renewcommand{\left}{\mathopen{}\mathclose\bgroup\originalleft}
\renewcommand{\right}{\aftergroup\egroup\originalright}

\makeatletter
\newcases{lrcases*}
  {\quad}
  {$\m@th{##}$\hfil}
  {{##}\hfil}
  {\lbrace}
  {\rbrace}
\makeatother

\usepackage[
    deletedmarkup=sout,
    commentmarkup=uwave,
    authormarkuptext=name
]{changes}

\title{Covering multigraphs with bipartite graphs}

\author{
Jaehoon Kim\thanks{Department of Mathematical Sciences, KAIST, South Korea. Email: {\tt jaehoon.kim@kaist.ac.kr, hyunwoo.lee@kaist.ac.kr}. Supported by the POSCO Science Fellowship of POSCO TJ Park Foundation and the National Research Foundation of Korea (NRF) grant funded by the Korea government(MSIT) No. RS-2023-00210430.}
\and
Hyunwoo Lee\footnotemark[1] \thanks{Extremal Combinatorics and Probability Group
(ECOPRO), Institute for Basic Science (IBS). Supported by the Institute for Basic Science (IBS-R029-C4).}
}

\date{\today}

\begin{document}

\maketitle
\begin{abstract}
Hansel's lemma states that $\sum_{H\in \mathcal{H}}|H| \geq  n \log_2 n$ holds where $\mathcal{H}$ is a collection of bipartite graphs covering all the edges of $K_n$. We generalize this lemma to the corresponding multigraph covering problem and the graphon covering problem. We also prove an upper bound on $\sum_{H\in \mathcal{H}}|H|$ which shows that our generalization is asymptotically tight in some sense.
\end{abstract}

\section{Introduction}\label{sec:intro}

Twenty questions is a popular game identifying an object with twenty yes/no questions. More generally, one can easily see that $\lceil \log n \rceil$ yes/no questions are required and sufficient to identify an object among the set $V$ of $n$ given objects (every $\log$ in this paper is the logarithm to the base two).
This simple-looking game captures the basic idea on how a computer stores data and efficiently searches them.

When storing data, we can record the partition $(S_i,T_i)$ of $V$ where $S_i$ is the set of objects whose answer to the $i$-th question is yes and $T_i$ is the set of objects whose answer to the $i$-th question is no.
Consider how much memory space we have to use in the data storage. It depends on specific ways of storing the data, but we assume that we use the memory space of size $|S_i\cup T_i|$ to remember one pair $(S_i,T_i)$. Then we use memory space of size $n\lceil \log n\rceil$ in total for a given set of $\lceil \log n\rceil$ pairs. 
At this point, we naturally come to a question that whether we can save memory space by making some modifications.

We consider that $i$-th question `distinguishes' the objects in $S_i$ from $T_i$, and we wish to `distinguish' all objects from other objects by the given set of pairs. 
In order for our purpose of saving memory space, we do not force the set $S_i\cup T_i$ to be the entire $V$. In other words, instead of considering partitions of $V$, we consider the pairs $(S_1,T_1),\dots, (S_m, T_m)$ where $S_i,T_i$ are disjoint subsets of $V$ and wishes that all pairs $x,y$ of distinct elements in $V$ are `distinguished' by at least one pair $(S_i,T_i)$. 
R\'{e}nyi~\cite{renyi1961random} in 1961 captured this problem by introducing the following concept of \emph{weakly separating system}: a collection of pairs of disjoint subsets $\mathcal{P} = \{ (S_1,T_1),\dots, (S_m,T_m)\}$ is an \emph{weakly separating system} of $V$ if, for every distinct $x,y\in V$, there exists $i\in [m]$ such that $x\in S_i, y\in T_i$ or $x\in T_i, y\in S_i$.
A \emph{strongly separating system} is also defined as follows: for every distinct $x,y\in V$, we can find $i,j\in [m]$ so that $x\in S_i, y\in T_i$ and $x\in T_j, y\in S_j$.
The weakly separating systems are closely related to some variations of the hash functions and graph theory problems, so they are extensively studied, for examples, see~\cite{bollobas2007separating, bollobas2007separating2, fredman1984size, hansel1964nombre, katona1966separating, katona1967problem, kundgen2001minimal, nilli1994perfect, radhakrishnan2003entropy, tuza1994applications, wegener1979separating}.

 A weakly separating system $\mathcal{P} = \{ (S_1,T_1),\dots, (S_m,T_m)\}$ of $V$ can be considered as a collection of complete bipartite graphs with vertex partitions $(S_i,T_i)$, which together covers all edges of the complete graph $K_n$ on the vertex set $V$.
 Let $\mathcal{H}=\{H_1,\dots, H_m\}$ be a \emph{bipartite covering} of $G$ if each graph $H_i$ is a bipartite graph with $V(H_i)\subseteq V(G)$ and every edge $e$ of $G$ belongs to $E(H_i)$ for some $i\in [m]$. Let $\capa(\mathcal{H}) = \sum_{i\in [m]}|H_i|$ be the \emph{capacity} of $\mathcal{H}$. 
 Hence, our original question about the memory space becomes the question of minimizing  the capacity $\capa(\mathcal{H})$ for a bipartite covering $\{H_1,\dots, H_m\}$ of $K_n$. Note that allowing the bipartite graphs $H_i$ to be a non-complete bipartite graph does not change the problem.

Hansel~\cite{hansel1964nombre} and independently Katona and Szemer\'{e}di~\cite{katona1967problem} proved the following lemma answering the question, that even after allowing $S_i\cup T_i$ to be not the entire set $V$, we cannot further save memory spaces.
\begin{lemma}[Hansel's lemma] \label{lem:Hansel_lemma}
Let $\mathcal{H}$ be a bipartite covering of $K_n$. Then we have $\capa(\mathcal{H}) \geq n\log n$
\end{lemma}
Note that this lemma is tight up to an additive constant. The precise minimum was computed by Bollob\'{a}s and Scott ~\cite[Theorem~3]{bollobas2007separating} by demonstrated that if $n = 2^k + \ell < 2^{k+1}$, then the minimum is $nk + 2\ell$.
 Katona and Szemer\'{e}di~\cite{katona1967problem} further generalized this lemma by considering a bipartite covering of an arbitrary graph.
\begin{lemma}\label{lem:katona-szemeredi}
    Let $G$ be a $n$-vertex simple graph. Let $\mathcal{H}=\{H_1, \dots, H_m\}$ be a bipartite covering of $G$. Then $\capa(\mathcal{H}) \geq \sum_{v\in V(G)} \log \left(\frac{n}{n-d(v)}\right) \geq n\log \left(\frac{n^2}{n^2 - 2e(G)} \right)$.
\end{lemma}

Arguably the most natural next question is to consider problems for non-ideal situations.
What if some of the twenty questions were answered incorrectly? In a more realistic scenario, what if noises are introduced, so some of the answers we obtain from the data storage are altered, and thus we obtained wrong answers?
In order to be able to tolerate some level of noise, one might want to ensure to distinguish the pairs $x,y \in V$ more than once. If our set of questions distinguishes $x$ from all other elements at least $\lambda$ times, then we are able to tolerate at most $\lambda-1$ lies.
This motivates the following generalization of the bipartite covering. 
Let $\mathcal{H}=\{H_1,\dots, H_m\}$ be a \emph{bipartite covering} of a multigraph $G$ if each graph $H_i$ is a bipartite graph with $V(H_i)\subseteq V(G)$ and for an edge $e\in E(G)$ with multiplicity $\lambda$, there are at least $\lambda$ disinct indices $i_1,\dots, i_{\lambda}\in [m]$ such that $e\in E(H_{i_j})$ for all $j\in [\lambda]$.

With this definition, our question is to find the minimum of $\capa(\mathcal{H})$ for a bipartite covering $\mathcal{H}$ of the complete multi-graph $K_n^{\lambda}$ with every edge having multiplicity $\lambda$. Note that a similar question was considered in~\cite[Chapter 14.2]{alon2016probabilistic} in terms of finding the minimum of $m= |\mathcal{H}|$ when all the bipartite graphs have to be spanning (i.e. $S_i\cup T_i=V$ for all $i$.).

Of course, we can take a bipartite covering of $K_n$ and duplicate them $\lambda$ times to obtain a bipartite covering $\mathcal{H}$ of $K_n^{\lambda}$ with $\capa(\mathcal{H})= \lambda n\log n + O(\lambda)$, but as we will see in \Cref{thm:graph_upperbound}, it is not efficient at all. Then how much can we minimize $\capa(\mathcal{H})$? For the lower bound of $\capa(\calH)$, we prove the following theorem generalizing Hansel's Lemma.


\begin{theorem}\label{cor:graph_lowerbound}
        Suppose $n, \lambda \in \mathbb{N}$. 
        Let $\mathcal{H}$ be a bipartite covering of the multigraph $K_n^{\lambda}$, then we have 
        $$\capa(\mathcal{H}) \geq \max\left\{2 \lambda (n-1), \enspace n\left( \log n + \lfloor \frac{\lambda - 1}{2} \rfloor \log \left(\frac{\log n}{\lambda} \right) - \lambda -1 \right)\right\}.$$
\end{theorem}
Note that the first term $2 \lambda(n-1)$ is larger if $\lambda \geq \frac{1}{2}\log n$ and the second term is much larger if $\lambda = o(\log n)$. In fact, if $\lambda = O(\log^{1-\ve} n)$, then we have $\capa(\mathcal{H}) = n\log n + \Omega(n \lambda \log\log n)$. 
In fact, the next theorem supplies an upper bound on $\capa(\mathcal{H})$ showing that $\capa(\mathcal{H}) = n\log n + \Theta(n \lambda \log\log n)$.

\begin{theorem} \label{thm:graph_upperbound}
 There exist a bipartite covering $\mathcal{H}= \{H_1,\dots, H_m\}$ of the multigraph $K_{n}^{\lambda}$ satisfying
    $$\capa(\mathcal{H}) \leq n\big( \log (n-1) + (1+ o(1))\lambda \log \log n\big).$$
    If $\lambda = k\log (n-1)$ with $k\geq 1/2$, then we can further ensure $|\mathcal{H}|\leq (2+ \frac{30}{\sqrt{k}}) \lambda n$.
\end{theorem}

This theorem provides that the two bounds on \Cref{cor:graph_lowerbound} determine the asymptotic growth rate of the minimum capacity of a bipartite covering of $K_n^{\lambda}$ when $\lambda$ is either $O(\log^{1-\ve} n)$ or $\omega(\log n)$.

We deduce \Cref{cor:graph_lowerbound} from a more general theorem regarding graphons. In order to describe our results, we briefly introduce what graphons are.  A function $W: [0, 1]^2 \to [0, 1]$ is a \emph{graphon} if it is a symmetric and measurable function. Graphons were introduced by Lov\'{a}sz and Szegedy~\cite{lovasz2006limits} and developed by Borgs, Chayes, Lov\'{a}sz, S\'{o}s and Vesztergombi~\cite{borgs2008convergent}. A space of graphons induced with a specific metric, so-called cut-distance, has various useful properties, hence it is very useful for studying dense graphs and sequences of dense graphs.  

For a graph $H$ on vertex set $[n]$, we consider the graphon $W_{H}$ such that $W_{H}(x,y) = 1$ if $x\in [\frac{i-1}{n},\frac{i}{n})$ and $y\in [\frac{j-1}{n},\frac{j}{n})$ for some $ij\in E(H)$ and $W_{H}(x,y)=0$ otherwise. This graphon captures various properties of $H$, showing that the graphons are natural generalizations of graphs.

As graphons are real-valued functions, we can define a uniform norm on graphons.
We write $m(A)$ to indicate the Lebesgue measure of $A$.
\begin{definition}
    For a graphon $W$, we define $L_{\infty}$-norm $\lVert \cdot \rVert _{\infty}$ as follows: $$\lVert W \rVert _{\infty} \defeq \inf \{\alpha \in [0, 1] : m(\{(x, y)\in [0, 1]^2 : W(x, y) > \alpha\}) = 0\}.$$
    \end{definition}

Note that for a graphon $W$, the inequality $\lVert W \rVert_{\infty} \leq \delta$ means that $W(x,y)$ is at most $\delta$ except for a set of points $(x,y)$ which has measure zero. We also define the following concept of bipartite graphons. The definition of bipartite graphon is the following. 

\begin{definition}
    A graphon $W$ is \emph{bipartite} if there exist two measurable sets $A, B\subseteq  [0, 1]$ such that $m(A\cap B) = 0$ and $m(supp(W)\setminus ((A\times B) \cup (B\times A)) ) = 0$. In such a case, we say that $W$ is $(A, B)$-bipartite.
\end{definition}

Note that if $W$ is a $(A, B)$-bipartite graphon, then $m(A)+m(B) = m(A\cup B) \leq 1$.
In order to generalize the bipartite coverings of complete graphs, we define the following concept.
We say that  a countable set $\mathcal{W}$ of graphons is a \emph{separating system} if each $W \in \mathcal{W}$ is a bipartite graphon. By giving an ordering on the set $\mathcal{W}$, we identify the separating system with a sequence $W_1,W_2,\dots $ of bipartite graphons.
In addition, we call it an \emph{$\eta$-full separating system} if 
$$\int_{[0,1]^2} \calW^{\Sigma}(x, y) dxdy \geq \eta,$$
where $\calW^{\Sigma}(x, y) \defeq \min\{1, \enspace \sum_{W\in \calW} W(x, y)\}$ for all $(x,y)\in [0,1]^2$.
For a given bipartite graphon $W$, we can choose a pair of disjoint measurable sets $(A_W,B_W)$ such that $m(A_W\cup B_W)$ is minimum among all choices where $W$ is $(A_W,B_W)$-bipartite. We assume that such a choice is taken for all graphon $W$ and we will say that $\sum_{W\in \mathcal{W}} m(A_W \cup B_W)$ is the \emph{capacity} of the separating system $\mathcal{W}$.
We are now ready to state our theorem regarding graphon separating systems.

\begin{theorem} \label{thm:main_noweight}
Suppose $0<\ve \leq 1$ and $\mathcal{W}$ is a $(1-\ve)$-full separating system where each $W_i\in \mathcal{W}$ is a $(A_i,B_i)$-bipartite graphon. Then we have $\sum_{k = 1}^{|\mathcal{W}|} m(A_k \cup B_k) \geq \log \frac{1}{\ve}$
\end{theorem}

Note that a bipartite covering $\mathcal{H}$ of a complete graph $K_n$ provides a $(1-\frac{1}{n})$-full separating system, as $\int_{[0,1]^2}\calW^{\Sigma}(x,y) dxdy = 1-\frac{1}{n}$ holds for $\calW=\{ W_H: H\in \mathcal{W}\}$. 
Hence \Cref{thm:main_noweight} implies 
\Cref{lem:Hansel_lemma}.
Moreover, we prove the following weighted version as well.
\begin{theorem} \label{thm:main_weight}
    Suppose $0<\ve <1$ and $0<\delta \leq 1$. Suppose that $\mathcal{W}$ is a $(1-\ve)$-full separating system where each $W_i\in \mathcal{W}$ is   $(A_i, B_i)$-bipartite and $\|W\|_{\infty}\leq \delta$.
 Then we have
 $$\sum_{k = 1}^{|\mathcal{W}|} m(A_k \cup B_k) \geq \max\left\{\frac{2(1-\ve)}{\delta} ,   \enspace  
 \log\frac{1}{\ve} + \lfloor \frac{1-\delta}{2\delta} \rfloor \log\left(\delta \log\frac{1}{\ve} \right) - \frac{1}{\delta} -1\right\}.$$
\end{theorem}
Again, the first term $\frac{2(1-\ve)}{\delta}$ is bigger if $\delta \leq  2( \log \frac{1}{\ve})^{-1})$.
Again, we can turn a bipartite covering $\mathcal{H}$ into a separating system $\mathcal{W}=\{ W_H: H\in  \mathcal{H}\}$, then Theorem~\ref{thm:graph_upperbound} implies that Theorem~\ref{thm:main_weight} is asymptotically tight for all $\delta=o( (\log \frac{1}{\ve})^{-1} )$ and $\delta = \Omega( (\log \frac{1}{\ve})^{ -1+\ve} ) $.


As graphons generalize multigraphs, Theorem~\ref{thm:main_weight} implies \Cref{cor:graph_lowerbound}. Indeed, for a given bipartite covering $\{H_1,\dots, H_m\}$ of $K^{\lambda}_n$, it is easy to see that $\{ \frac{1}{\lambda} W_{H_1},\dots, \frac{1}{\lambda}W_{H_m}\}$ is a $(1-\frac{1}{n})$-full separating system of graphons. Moreover, \Cref{thm:main_weight} implies the following lemma, which is a generalization of Lemma~\Cref{lem:katona-szemeredi} and it implies \Cref{cor:graph_lowerbound}.

\begin{lemma}\label{thm:edge-density-hansel}
    Let $G$ be a $n$-vertex multigraph with maximum multiplicity $\lambda$ and let $T= \frac{\lambda n^2}{\lambda n^2 - 2e(G)}.$ 
    Let $\mathcal{H}=\{ H_1, \dots, H_m\}$ be a bipartite covering of $G$. Then we have 
    \begin{equation*}
        \capa(\mathcal{H})  \geq \max\left\{ 2(1-\frac{1}{T})\lambda n ,\enspace n\left(\log T + \lfloor \frac{\lambda-1}{2} \rfloor\log \left(\frac{\log T}{\lambda} \right) - \lambda-1 \right)\right\}.
    \end{equation*}
\end{lemma}
\begin{proof}
Consider $\mathcal{W}=\{ \frac{1}{\lambda} W_{H_1},\dots, \frac{1}{\lambda} W_{H_m}\}$, then we have a natural bipartition $(A_i,B_i)$ for $W_{H_i}$ where $|H_i| = n\cdot m(A_i\cup B_i)$.
Then every $W\in \mathcal{W}$ satisfies $\|W\|_{\infty}\leq 1/\lambda$ and they are all bipartite graphons. 
Moreover, we have 
$$ \int_{[0, 1]^2} \calW^{\Sigma} dX \geq \frac{2e(G)}{\lambda n^2},$$
hence $\mathcal{W}$ forms a $(1-\frac{1}{T})$-full separating system. 
Thus \Cref{thm:main_weight} implies the desired inequality.
\end{proof}

In \Cref{sec:prelim}, we collect some preliminaries. 
In \Cref{sec:upper}, we prove upper bounds showing that our main results are asymptotically tight for $\lambda$ is either $O(\log^{1-\ve} n)$ or $\omega(\log n)$.
Two main theorems \Cref{thm:main_noweight} and \Cref{thm:main_weight} will be proven in \Cref{sec:proof}. Finally, in \Cref{sec:geometry}, we present an application of our results on some problems in discrete geometry.

\section{Preliminaries}\label{sec:prelim}
In this section, we want to collect some notations.
For $A\subseteq [0,1]^2$, we write $\mathbf{1}_{A}$ to denote the characteristic function of $A$, which satisfies $\mathbf{1}_{A}(x,y)=1$ if $(x,y)\in A$ and $\mathbf{1}_{A}(x,y)=0$ otherwise.
As graphons are generalizations of graphs, we can generalize many graph parameters to graphons.
Since every graphon $W$ is a measurable function that is bounded by the integrable function $\mathbf{1}_{[0, 1]^2}$, $W$ is also integrable. Thus, by Fubini's theorem, the degree function $d_W(x)$ defined below is defined almost everywhere.
\begin{definition}
    Let $W$ be a graphon. Then for each $x\in [0, 1]$, we denote $d_W (x) \defeq \int_{[0, 1]} W(x, y) dy$ as a degree function of $W$ at $x$.
\end{definition}

For a given countable collection $\mathcal{W}$ of graphons, we say that it is \emph{bounded above by $\delta$} if every $W\in \calW$ satisfies $\|W\|_{\infty} \leq \delta$.
We define a function $\nw: [0,1]\to \mathbb{N}\cup \{\infty\}$ as $\nw(x) \defeq  |\{W\in \calW: x\in\ A_W\cup B_W\}|.$ If a sum  $\sum_{i=1}^{\infty} f(i)$ of nonnegative numbers diverges to infinity, then we write $\sum_{i=1}^{\infty} f(i) = \infty.$ 
Then the following lemma holds, where $\int_{[0, 1]} \nw(x) dx$ is $\infty$ if (but not only if) $\nw(x) =\infty$ for positive measure of points $x\in [0,1]$.
As it is easy to see that the function $\nw(x)$ is a measurable function, Fubini's theorem implies that the following lemma holds.

\begin{lemma} \label{lem:double_counting}
    Let $\calW$ be a separating system of graphons. Then $\sum_{W\in \calW} m(A_W\cup B_W) = \int_{[0, 1]} \nw(x) dx.$ 
\end{lemma}
    

With these definitions and \Cref{lem:double_counting}, we can investigate the value of $\int_{[0, 1]} \nw(x) dx$ instead of $\sum_{W\in \calW} m(A_W\cup B_W)$.

\section{Existence of asymptotically optimal bipartite coverings of $K_n^{\lambda}$}\label{sec:upper}

In this section, we prove \Cref{thm:graph_upperbound}. For a given multigraph $G$ and an edge $e=xy \in \binom{V(G)}{2}$, we write $\mu_G(e)$ be the multiplicity of the edge $e$. For $x, y\in V(G)$, we consider $\mu_G(xy) = 0$ if and only if $xy \notin E(G).$ We simply write $\mu(e)$ if $G$ is clear from the context.

For two multigraphs $G$ and $H$, we write $G-H$ to denote the multigraph such that an edge $e \in \binom{V(G)}{2}$ is in $G-H$ with multiplicity $\mu_{G-H}(e) = \max\{ \mu_G(e)-\mu_H(e), 0\}.$ 

Let $\alpha$ be a real number. For a multigraph $G$ and a vertex $v\in V(G)$, we define the $\alpha$-exponential degree of a vertex $v$ as
$$d^{\rm ex}_G(\alpha,v) = \sum_{u\in N(v)} \alpha^{\mu(uv)}.$$
Note that if $w$ is not adjacent to $v$, then the pair $vw$ contributes zero to the exponential degree of $v$. By this definition, degree of $v$ is same as $1$-exponential degree of $v.$
Let $\Delta^{\rm ex}(\alpha, G)\defeq \max\{ d^{\rm ex}(\alpha, v): v\in V(G)\}$ be the \emph{maximum $\alpha$-exponential degree} of $G$. Note that any graph with at least one edge has the maximum exponential degree at least $\alpha$.

By modifying Erd\H{o}s' proof on max cut in~\cite{erdos1965some}, we can prove the following lemma.
\begin{lemma}\label{lem:erdos_multi_dense_bipartite}
Let $\alpha > 1$.  For every multigraph $G$, there is a simple bipartite subgraph $H$ of $G$ such that for every $v\in V(G)$, 
    $$d^{\rm ex}_{G-H}(v) \leq \frac{1}{2}(1+\alpha^{-1})d^{\rm ex}_G(v).$$ 
    In particular, $\Delta^{\rm ex}(G-H) \leq \frac{1}{2}(1+\alpha^{-1})\Delta^{\rm ex}(G).$
\end{lemma}
\begin{proof}
    Let $A, B$ be a partition of $V(G)$ which maximizes $\sum_{e\in E[A, B]} \alpha^{\mu(e)}$ in $G$. 
    Let $H$ be the simple graph with edges in $\{ e: e\in E(G[A, B])\}.$ 
    
    The supposed maximality ensures that every $v\in A$ and $u\in B$ satisfies  
    $$\sum_{w\in N_G(v)\cap B} \alpha^{\mu(vw)} \geq \frac{1}{2}d^{\rm ex}_G(v) \enspace \text{ and }\enspace  \sum_{w\in N_G(u)\cap A} \alpha^{\mu(uw)} \geq \frac{1}{2}d^{\rm ex}_G(u).$$ 
    Hence, for each vertex $v$, we have 
    $$d^{\rm ex}_{G-H}(v) = \sum_{u\in N_G(v)\cap A} \alpha^{\mu(uv)} + \sum_{u\in N_{G}(v)\cap B} \alpha^{\mu(uv) - 1}\geq  (\frac{1}{2}+ \frac{1}{2}\alpha^{-1}) d^{\rm ex}_G(v).$$
    Similarly, we also have $d^{\rm ex}_{G-H}(u) \geq \frac{1}{2}(1+\alpha^{-1})d^{\rm ex}_G(v)$ for all $u\in B$. This proves the lemma.
\end{proof}
By repeatedly applying the above lemma, we obtain the following.

\begin{lemma}\label{lem:cover-bipartite-maxdeg}
   Let $\alpha >1$. Let $G$ be a multigraph with ${\Delta}^{\rm ex}(\alpha, G) = {\Delta}^{\rm ex} \geq \alpha.$ Then, there exist a bipartite covering $\mathcal{H}=\{H_1,\dots, H_m\}$ of $G$ with $m \leq \frac{\log \Delta^{\rm ex} }{1-\log(1+\alpha^{-1})}.$
\end{lemma}
\begin{proof}
    Let $G_0 = G$. 
    For given $G_i$ with  ${\Delta}^{\rm ex}(G_i) \leq  (\frac{1+\alpha^{-1}}{2})^i{\Delta}^{\rm ex}$, we apply \Cref{lem:cover-bipartite-maxdeg} to choose $H_i$ so that $G_{i+1}=G_i - H_i$ satisifes 
    ${\Delta}^{\rm ex}(G_i) \leq  (\frac{1+\alpha^{-1}}{2})^{i+1}{\Delta}^{\rm ex}$.
    Repeating this, then we will have $\Delta^{\rm ex}(G^{m+1})<\alpha$ for some $m\leq \log_{2/(1+\alpha^{-1})}\Delta^{\rm ex} = \frac{\log \Delta^{\rm ex} }{1-\log(1+\alpha^{-1})},$
    implying that $G^m$ has no edges. Thus the set $\{H_1,\dots, H_m\}$ is the desired bipartite covering of $G$.
\end{proof}

We now prove \Cref{thm:graph_upperbound}.
\begin{proof}[Proof of \Cref{thm:graph_upperbound}]
Note that $\Delta^{\rm ex}(\alpha, K_{n}^{\lambda}) = \alpha^{\lambda} (n-1)$ for any $\alpha>1$. 
As each bipartite graph has at most $n$ vertices, we only need to find a bipartite covering $\mathcal{H}$ with size 
$\log (n-1) + (1+ o(1))\lambda \log \log n$ and size $(2+ \frac{40}{\sqrt{k}}) \lambda$ in the each case.

First, consider $\alpha = \log n$. Then \Cref{lem:cover-bipartite-maxdeg} implies that there are bipartite covering $\mathcal{H}$ satisfying the following.
$$|\mathcal{H}| \leq \frac{ \log \Delta^{\rm ex}}{1- \log(1+\alpha^{-1}) } \leq \frac{ \lambda \log \alpha + \log (n-1)}{ 1- \log(e)/\alpha } \leq (1+ \frac{2\log e}{\log n})\left (\log(n-1) + \lambda \log\log n\right),$$
which implies the first bound.

Now, consider the case where $\lambda = k\log (n-1)$ with $k\geq 4$. 
Let $0<\ve \leq \frac{1}{2}$ and $\alpha = \frac{1}{1-\ve}$.
Then \Cref{lem:cover-bipartite-maxdeg} implies that there are bipartite covering $\mathcal{H}$ satisfying the following.
$$|\mathcal{H}| \leq \frac{ \log \Delta^{\rm ex}}{1- \log(1+\alpha^{-1}) } \leq \frac{ \lambda \log(\frac{1}{1-\ve}) + \log (n-1)}{ 1- (1- \frac{\log(e)\ve}{2+\ve}) } \leq \frac{ \log (e) (\ve+\ve^2)\lambda + \log(n-1)}{ \frac{\log(e)\ve}{2+\ve} }
\leq  (2+5\ve) \lambda + \frac{2}{\ve} \log(n-1)
$$
The second and third inequality follows from the taylor expansion of $\log(2+x)$ at $x=0$ and the taylor expansion of $\log(1-x)$ at $x=0$.
By taking $\ve = k^{-1/2}$, we obtain a bipartite covering $\mathcal{H}$ with $|\mathcal{H}| \leq (2+ \frac{7}{\sqrt{k}}) \lambda.$ 
In particular, this shows that there exists a bipartite covering $\mathcal{H}$ wth $|\mathcal{H}| = (2+ \frac{7}{2})4\log(n-1)= 22\log(n-1)$.
Note that a bipartite covering for $K^{4\log (n-1)}_n$ is also a bipartite covering for $K^{\lambda}_n$ for $\lambda \leq 4\log(n-1)$.
Thus, for $1/2\leq k\leq 4$, we have $(2+\frac{30}{\sqrt{k}})k\log(n-1)\leq 22\log(n-1)$.
Thus we obtained a desired biparatite covering.
This finishes the proof.
\end{proof}
\begin{rmk}\label{remark1}
Note that if $\lambda = \log^c n$ for some $\frac{1}{2}\leq c\leq 1$, then we can improve the above bound to $\capa(\mathcal{H}) \leq n\big( \log (n-1) + (c+ o(1))\lambda \log \log n\big)$. Indeed, taking $\alpha = \log^c n$ in the above proof yields this bound. 
Note that even after this improvement, there's still $O(\lambda n\log\log n)$-gap between the bound in \Cref{cor:graph_lowerbound} and the bound in \Cref{thm:graph_upperbound}.
\end{rmk}

\section{Proof of \Cref{thm:main_noweight} and \Cref{thm:main_weight}}\label{sec:proof}

In this section, we prove the \Cref{thm:main_noweight} and \Cref{thm:main_weight}. 
For this, we first collect the following lemma providing a connection between the capacity of the bipartite covering of graphons and function $\nw$ we defined in Section~\ref{sec:prelim}. 

For our convenience, we define the following for a given $0<\delta\leq 1$ and $n\in \mathbb{N}$.
$$\ell_{\delta} := \lfloor \frac{1}{2\delta} \rfloor, \enspace \tau_{\delta} := \frac{1}{2} - \ell_{\delta} \delta.$$
\begin{lemma}\label{lem:graphon_lemma}
    Let $0<\ve<1$ and $\calW = \{W_1, \dots , W_n\}$ be a $(1-\ve)$-full finite separating system bounded above by $\delta > 0$ where each $W_i$ is $(A_i, B_i)$-bipartite. Then we have
    \begin{equation*}
        \sum_{i=1}^n m(A_i \cup B_i) \geq \log \frac{1}{\ve} + \int_{[0, 1]} \log \left(\sum_{j=0}^{\ld} \left(\frac{1}{2} + \td - \delta j \right) \binom{N_{\calW}(x)}{j}\right) dx.
    \end{equation*}
\end{lemma}

\begin{proof}
    For each $i\in [n]$, we pick $C_i\in \{A_i,B_i\}$ independently and uniformly at random.
    Let $\mathcal{C}=\{C_1,\dots, C_n\}$ be the collection of chosen sets. 
    We then assign for each $x\in [0,1]$, 
    $$\gamma(x)\defeq \frac{1}{2} - \delta|\{i\in [n]: x\in C_i \}|.$$ 
    Note that $\gamma$ is a random variable depending on the choice of $\mathcal{C}$.
    For every $(x, y)\in [0, 1]^2$, we denote 
     $${\rm sep}(x, y) \defeq \{i\in [n]: (x, y)\in (A_i\times B_i)\cup (B_i\times A_i)\}.$$ 
     For all $(x,y)\in [0,1]^2$ and each $i\in {\rm sep}(x,y)$, we have either $x\in C_i$ or $y\in C_i$, meaning that 
    $$|{\rm sep}(x, y)|\leq \left(|\{i\in [n]: x\in C_i \}| + |\{j\in [n]: y\in C_j \}|\right).$$

   Now, we consider a set $E\subseteq [0, 1]$ where $E = \{x\in [0, 1]: \gamma(x) \geq 0\}$, which is clearly a measurable set. 
As $\calW$ is bounded above by $\delta$, for almost all $(x,y)\in E\times E$,
we have 
$$\calW^{\Sigma}(x, y) \leq \delta |{\rm sep}(x, y)| \leq 1- \gamma(x)-\gamma(y).$$
Hence, for almost all $x\in E$, we have
$$\int_E \calW^{\Sigma}(x, y)dy \leq \int_{E} \delta |{\rm sep}(x, y)| dy
\leq m(E) - \gamma(x)m(E) - \int_E \gamma(y)dy \leq m(E) - (\gamma(x) + \td)m(E).$$
 The final inequality holds as $y\in E$ implies $w(y)\geq \tau_{\delta} = \frac{1}{2} -\lfloor \frac{1}{2\delta}\rfloor\delta$.
This implies that almost all $x\in E$ satisfies
\begin{align}\label{eq: 1}
    d_{\calW^{\Sigma}}(x) \leq \int_E \calW^{\Sigma}(x, y)dy + \int_{[0, 1]\setminus E} 1~ dy \leq 1 - (\gamma(x) + \td)m(E).
\end{align} 
We claim that we also have 
\begin{equation}\label{eqt:integral}
        \int_{E} \frac{\gamma(x) + \td}{1 - d_{\calW^{\Sigma}}(x)} dx \leq 1.
    \end{equation}
Indeed, if $m(E) = 0$, then $\int_{E} \frac{\gamma(x) + \td}{1 - d_{\calW^{\Sigma}}(x)} dx = 0 \leq 1$. If $m(E) > 0$, then \eqref{eq: 1} implies $\int_{E} \frac{\gamma(x) + \td}{1 - d_{\calW^{\Sigma}}(x)} dx \leq \int_{E} \frac{1}{m(E)} = 1$. Thus \eqref{eqt:integral} holds.

  Let 
  
  $$\gamma^*(x) \defeq \begin{cases}
        \gamma(x) + \td & \text{ if } {\gamma(x) \geq 0}\\
        0 & \text{ if } {\gamma(x) < 0}.
    \end{cases}$$
    By \Cref{eqt:integral}, 
    $$\int_{[0, 1]} \frac{\gamma^*(x)}{1 - d_{\calW^{\Sigma}}(x)} dx \leq 1.$$

For a fixed choice of $x$, let $\mathbb{E}[\gamma^*(x)]$ be the expectation of $\gamma^*(x)$ over the choice of $\mathcal{C}$. 
    As there are finitely many bipartite graphons in $\mathcal{W}$, the linearity of expectation together with the above displayed equation yields the following.
    $$\int_{[0, 1]}  \frac{\mathbb{E}[\gamma^*(x)]}{1 - d_{\calW^{\Sigma}}(x)} dx \leq 1.$$
    
For each $x\in [0, 1]$, we can compute $\mathbb{E}[\gamma^*(x)]$ as follows. 
$$\mathbb{E}[\gamma^*(x)] = \sum_{j = 0}^{\ld} \left(\frac{1}{2} + \td - \delta j \right) \binom{N(x)}{j} 2^{-N(x)}.$$
Thus we have 
$$\int_{[0, 1]}  \frac{1}{1 - d_{\calW^{\Sigma}}(x)} \sum_{j = 0}^{\ld} \left(\frac{1}{2} + \td - \delta j \right) \binom{N(x)}{j} 2^{-N(x)} dx \leq 1. $$
By taking logarithms on both sides and applying Jensen's inequality, this yields the following.
$$\int_{[0, 1]} -\log(1-d_{\calW^{\Sigma}}(x)) dx + \int_{[0, 1]} -N_{\calW}(x)dx + \int_{[0, 1]} \log \left(\sum_{j = 0}^{\ld} \left(\frac{1}{2} + \td - \delta \right) \binom{N_{\calW}(x)}{j} \right)dx \leq 0.$$
    By \Cref{lem:double_counting}, we have $\int_{[0, 1]} N_{\calW}(x)dx = \sum_{i=1}^k m(A_i \cup B_i)$ and Jensen's inequality yields $\int_{[0, 1]}-\log (1 - d_{\calW^{\Sigma}}(x)) dx \geq -\log \left(\int_{[0, 1]} 1 - d_{\calW}(x) dx \right) \geq \log \frac{1}{\ve}$. Therefore, we obtain the following desired inequality.
    $$\sum_{i=1}^n m(A_i \cup B_i) \geq \log \frac{1}{\ve} + \int_{[0, 1]} \log \left(\sum_{j=0}^{\ld} \left(\frac{1}{2} + \td -\delta j\right) \binom{N_{\calW}(x)}{j} \right)dx.$$
\end{proof}

We are now ready to prove our main theorems. We first prove \Cref{thm:main_noweight}.
\begin{proof}[Proof of \Cref{thm:main_noweight}]
    Let $\calW$ be a separating system where each $W_i$ is $(A_i,B_i)$-bipartite. If $\calW$ is finite, we can add zero graphons $W_i$ (which is $(\emptyset, \emptyset)$-bipartite) for all $i>|\mathcal{W}|$ to assume $|\calW|=\infty$.
    For each $n\in \mathbb{N}$, denote $\calW_n \defeq \{W_1, \dots, W_n\}$ and assume that $\calW_n$ is $(1-\ve_n)$-full bipartite covering. Then we have $\lim_{n\to \infty} \ve_n \leq \ve$. 
    For $\delta=1$, we have $\ell_{\delta}=0$ and $\tau_{\delta}=\frac{1}{2}$. Hence 
    \Cref{lem:graphon_lemma} implies that for each $n\in \mathbb{N}$, we have $$\sum_{i=1}^n m(A_i \cup B_i) \geq \log \frac{1}{\ve_n} + \int_{[0, 1]} \log \left(\sum_{j=0}^{0} \left(\frac{1}{2} + \frac{1}{2} -\delta j\right) \binom{N_{\calW_n}(x)}{j} \right)dx =  \log \frac{1}{\ve_n}.$$ 
Therefore, we conclude $\sum_{i=1}^{\infty} m(A_i\cup B_i) \geq \lim_{n\to \infty} \log \frac{1}{\ve_n} \geq \log \frac{1}{\ve}$.
\end{proof}

Finally, we proceed to the proof of \Cref{thm:main_weight}.
\begin{proof}[Proof of \Cref{thm:main_weight}]
By adding zero graphons if needed, assume that $|\mathcal{W}|=\infty$. Let $\calW = \{W_1, W_2, \dots \}$ and for each $i\in \mathbb{N}$, let $W_i$ is $(A_i, B_i)$-bipartite.
As $\mathcal{W}$ is a $(1-\ve)$-full covering, $\int_{[0,1]^2} \mathcal{W}^{\Sigma}(x,y) dxdy \geq 1-\ve$. As each $W_i\in \mathcal{W}$ is $(A_i, B_i)$-bipartite and bounded above by $\delta$, Mantel's theorem yields that $\int_{[0,1]^{2}} W_i(x,y)dx \leq \frac{\delta}{2}m(A_i\cup B_i)^2$.
Then we have 
$$1-\ve \leq \int_{[0,1]^2} \mathcal{W}^{\Sigma}(x,y) dxdy 
\leq \sum_{i = 1}^{\infty} \frac{\delta}{2}m(A_i\cup B_i)^2 \leq \frac{\delta}{2} \sum_{i = 1}^{\infty} m(A_i\cup B_i).$$
This provides $\sum_{i = 1}^{\infty} m(A_i\cup B_i) \geq \frac{2(1-\ve)}{\delta}$.

Now it suffices to prove that $A=\log \frac{1}{\ve} + \lfloor \frac{1-\delta}{2\delta} \rfloor \log\left( \delta \log \frac{1}{\ve} \right) - \frac{1}{\delta} - 1$ is the lower bound on $\sum_{i=1}^{\infty} m(A_i\cup B_i)$.

If $\frac{1}{\delta} > \frac{1}{2} \log \frac{1}{\ve}$, then $\frac{2(1-\ve)}{\delta}$ is bigger than $A$. Hence we assume $\frac{1}{\delta} \leq \frac{1}{2} \log \frac{1}{\ve}$.
Assume that we have an an ordering $W_1,W_2,\dots$ of the graphons in $\mathcal{W}$ and assume that a graphon $W_i$ is $(A_i, B_i)$-bipartite for each $i\in \mathbb{N}$. 
In addition, we may assume that each $W_i$ is a step graph
$$W_i = \delta \cdot \mathbf{1}_{(A_i\times B_i) \cup (B_i\times A_i)}.$$
For each $n\in \mathbb{N}$, we denote $\calW_n \defeq \{W_1, \dots, W_n\}$. Let $\int_{[0, 1]^2} \calW_n^s dX = 1-\ve_n$, then $\lim_{n\to \infty} \ve_n \leq \ve$. 
Hence we can take an integer $N$ such that $\ve_n<1$ 

Fix an integer $n>N$. 
Let $E=\{ x\in [0,1]: N_{\calW_n}(x) < \frac{1}{2}\log\frac{1}{\ve_n} \}$. 
Note that $E$ can be written as a finite union of finite intersections of the complements of $(A_i\cup B_i)$, which are measurable sets. Hence, $E$ is measurable.
We now show that the measure $m(E)$ of the set $E$ is small.

\begin{claim}\label{cl:1}
$m(E) < \ve_n^{\frac{1}{5}}$.
\end{claim}
\begin{claimproof}
    By taking a subset if necessary, we assume $m(E) = \ve_n^{\frac{1}{5}}$. 
    Let $\ve'$ be a positive number satisfying $\ve' < \ve_n^c$ where $c$ is a positive integer where $\ve_n ^{\frac{1}{20}} + 4\ve_n^{c - \frac{3}{4}} < 1$ holds. Since $\ve_n < 1$ and $\ve_n^x$ goes to zero as $x$ goes to $\infty$, such an integer $c$ exists.
    
    As $E$ is measurable, we can find an open set $O\subseteq [0,1]$ containing $E$ with $m(O) = \ve_n^{1/5}+ \ve'$. Let $\eta = m(O)$, then we have $\eta<1$. As $O$ is an open set in $\mathbb{R}^1$, $O$ is a countable union of open intervals, hence there exists a natural measure-preserving bijection $\psi$ from $O$ to  $[0,\eta]\setminus B$ for some countable set $B$. (For example, map each of the countable intervals into the intervals $(0,x_1), (x_1,x_2),\dots$ in a way that the measure is preserved.)

    Observe that $$\int_{O^2} \calW_n^{\Sigma} dxdy + m( [0,1]^2- O\times O ) \geq \int_{[0,1]^2} \calW_n^{\Sigma} dxdy \geq 1-\ve_n,$$ 
    hence we have $\int_{O^2} \calW_n^{\Sigma} dxdy \geq \eta^2 - \ve_n$.
    Since 
    \begin{align*}
        \int_{E^2} \calW^{\Sigma}_n dxdy &= \int_{O^2} \calW^{\Sigma}_n dxdy - \int_{O^2 \setminus E^2} \calW^{\Sigma}_n dxdy \geq \int_{O^2} \calW^{\Sigma}_n dxdy - \int_{O^2 \setminus E^2} 1 dxdy\\ &\geq \int_{O^2} \calW^{\Sigma}_n dxdy - 2m(O\setminus E) m(E)- m(O\setminus E)^2
        =\int_{O^2} \calW^{\Sigma}_n dxdy - 2\ve' m(E)- \ve'^2 ,
    \end{align*} we have $\int_{E^2} \calW^{\Sigma}_n dxdy \geq \eta^2 - 2\ve' \ve_n^{\frac{1}{5}} - \ve'^2 - \ve_n$.

    For each $i\in [n]$, consider a new graphon $U_i$ such that 
    $$ U_i(x,y) = \left\{\begin{array}{ll}
        0 & \text{ if } \eta \cdot x \notin \psi(E) \text { or } \eta \cdot y\notin \psi(E) \\
        W_i(\psi^{-1}(\eta x), \psi^{-1}(\eta y))  & \text{ otherwise }
    \end{array}\right.$$
    In other words, $U_i$ is obtained from $W_i$ by moving its values on $O\times O$ into $[0,\eta]^2$ and dilate it into $[0,1]^2$. Let $\mathcal{U} = \{U_1,\dots, U_n\}$ and for each $i\in [n]$, let
    $$A'_i= \eta^{-1} \psi(E\cap A_i) \text{ and } B'_i = \eta^{-1}\psi(E\cap B_i),$$
    where the multiplication by $\eta^{-1}$ is the dilation of $[0,\eta]$ to $[0,1]$. Then each $U_i$ is $(A'_i,B'_i)$-bipartite.
Then we naturally have 
$\int_{[0,1]} U_i(x,y) dxdy = \eta^{-2} \int_{E^2} W_n(x,y) dxdy$, hence we have 
$$\int_{[0,1]} \mathcal{U}^{\Sigma} dxdy \geq \eta^{-2}(\eta^2 - 2\ve'\ve_n^{\frac{1}{5}} - \ve'^2 - \ve_n) \geq  1 - 3\ve_n^{c - \frac{1}{5}} - \ve_n^{\frac{3}{5}} \geq 1 - \ve_n^{\frac{11}{20}}.$$
Here, the final inequality holds by our choice of $c$.
Then  \Cref{thm:main_noweight} implies that
\begin{align}\label{eq: cl1}
    \sum_{i\in [n]} m(A'_i\cup B'_i) \geq  \log \ve_n^{-11/20} = \frac{11}{20} \log\frac{1}{\ve_n}.
\end{align}
However, as $N_{\mathcal{W}}(x)<\frac{1}{2}\log \frac{1}{\ve_n}$ for all $x\in E$, we have the following inequalities.
\begin{align*}
\sum_{i\in [n]} m(A'_i\cup B'_i) &= \int_{[0,1]} N_{\mathcal{U}}(x) dx 
 = \eta^{-1} \int_{E} N_{\mathcal{W}_n}(x) dx \leq \eta^{-1} m(E) \cdot \frac{1}{2}\log \frac{1}{\ve_n}
\leq \frac{1}{2} \log \frac{1}{\ve_n}.
\end{align*}
This contradicts \eqref{eq: cl1}. Hence we have $m(E) < \ve_n^{1/5}$ and this finishes the proof.
\end{claimproof}
    
Note that  \Cref{lem:graphon_lemma} implies the following inequality. 
$$\sum_{i=1}^n m(A_i \cup B_i) \geq \log \frac{1}{\ve_n} + \int_{[0, 1]} \log \left(\sum_{j=0}^{\ld} \left(\frac{1}{2} + \td -\delta j\right) \binom{N_{\calW_n}(x)}{j} \right)dx.$$
By the definition of $E$, we have bounds on the value of $N_{\mathcal{W}_n}(x)$ for all $x\notin E$.
Hence, 
    \begin{equation}\label{equ:...}
        \begin{split}
            \sum_{i=1}^n m(A_i \cup B_i) & \geq \log\frac{1}{\ve_n} + \int_{[0, 1]} \log \left(\sum_{j=0}^{\ld} \left(\frac{1}{2} + \td - \delta j \right) \binom{N_{\calW_n}(x)}{j}\right) dx\\
            & \geq \log\frac{1}{\ve_n} + \int_E \log \left(\frac{1}{2} + \td \right)dx + \int_{[0, 1]\setminus E} \log \left(\sum_{j=0}^{\ld} \left(\frac{1}{2} + \td - \delta j \right) \binom{\frac{1}{2}\log\frac{1}{\ve_n}}{j}\right) dx
        \end{split}
    \end{equation}
Note that the above term $\frac{1}{2} + \td - \delta j$ is at least $\frac{1}{2}\delta$ for $j\leq j^*=\lfloor \frac{1-\delta}{2\delta} \rfloor$. Hence we have 
\begin{align*}
\log \sum_{j=0}^{\ld} \left(\frac{1}{2} + \td - \delta j \right) \binom{\frac{1}{2}\log\frac{1}{\ve_n}}{j} &\geq \log \left(\frac{\delta}{2} \binom{\frac{1}{2}\log \frac{1}{\ve_n}}{j^*}\right) \geq \log \left(\frac{\delta}{2} \left( \frac{\log\frac{1}{\ve_n}}{2j^*} \right)^{j^*}\right) \\ &\geq \log\delta -1 + j^*(\log\log \frac{1}{\ve_n}  - \log (2j^*) ).
\end{align*}
Thus, \Cref{equ:...} and Claim~\ref{cl:1} yield the following.
 \begin{align*}\label{equ:...-middle}        
 \sum_{i = 1}^n m(A_i\cup B_i) &\geq \log\frac{1}{\ve_n} - m(E) + (1-m(E))( \log\delta -1 + j^*(\log\log \frac{1}{\ve_n}  - \log (2j^*) )) \\
        & \geq \log\frac{1}{\ve_n} -\ve_n^{1/5} + (1- \ve_{n}^{1/5} )( \log\delta -1 + j^*(\log\log \frac{1}{\ve_n}  - \log (2j^*) )) \\
        & \geq \log\frac{1}{\ve_n} + \lfloor \frac{1-\delta}{2\delta} \rfloor \log\log\frac{1}{\ve_n} - \lfloor \frac{1 - \delta}{2\delta}\rfloor \log{\frac{1}{\delta}} - \frac{1}{\delta} -1.
        \end{align*}
Note that since $0 < \ve_n \leq 1$, we have $\ve_n^{1/5} \log\log\frac{1}{\ve_n} \leq 2$, so the final inequality holds as all other negative terms are lower bounded by the last term $-\frac{1}{\delta}-1$.
As $\lim_{n\rightarrow \infty} \ve_n= n$, we take a limit on the above lower bound, then we obtain
$$\sum_{i=1}^{\infty} m(A_i\cup B_i) \geq \log\frac{1}{\ve} + \lfloor \frac{1-\delta}{2\delta} \rfloor \log \log \frac{1}{\ve} - \lfloor \frac{1 - \delta}{2\delta} \rfloor \log{\frac{1}{\delta}} - \frac{1}{\delta} -1 = \log\frac{1}{\ve} + \lfloor \frac{1-\delta}{2\delta} \rfloor \log\left(\delta \log\frac{1}{\ve} \right) - \frac{1}{\delta} -1.$$
This completes the proof.    
\end{proof}
\section{Applications}\label{sec:geometry}

We now present an application of \Cref{thm:main_noweight}.
Let $T$ be the right triangle with three sides of lengths $1, 1, \sqrt{2}$. Does there exist a countable collection of squares covering almost all points of $T$ such that (1) each square is completely contained in $T$ and (2) the sum of the lengths of the squares is finite? As we see in \Cref{obs:triangle_square_finite}, such a collection exists. 
For a given square $S$ in $\mathbb{R}^2$, let $\ell(S)$ be the length of the sides of $S$.

\begin{observation} \label{obs:triangle_square_finite}
    Let $T$ be a right triangle on the plane with sides having lengths $1, 1, \sqrt{2}$. Then there exists a countable collection $\calS$ of squares which covers almost all points of $T$ satisfying the following.
    \begin{itemize}
        \item Each $S\in \calS$ is completely contained in $T$, and
        \item $\sum_{S\in \calS} \ell(S) = 7 + 5\sqrt{2}$.
    \end{itemize}
\end{observation}
\begin{proof}
    Let $A, B, C$ be the three vertices of $T$ where $\angle ABC = \frac{\pi}{2}$. Let $D$, $E$ be trisection points of $AC$. Consider the square $S_1$ which has $DE$ with one side and is completely contained in $T$. Then we have $\ell(S_1) = \frac{\sqrt{2}}{3}$. Now consider a square $S'_1$ with side length $\frac{1}{3}$ and $B$ is a vertex of $S'_1$ and $S'$ is completely contained in $T$. 
   
   As $T\setminus (S_1\cup S'_1)$ is a disjoint union of two right angle triangles where each triangle is similar with $T$ with the ratio $\frac{\sqrt{2}}{3} : 1$. For each smaller triangle, iterate this process. Then we obtain a countable collection $\calS$ of squares covering almost all of $T$. Moreover 
   $$\sum_{S\in \calS} \ell(S) = \sum_{n\in \mathbb{N}} \left(\frac{2\sqrt{2}}{3}\right)^{n-1} \frac{\sqrt{2} + 1}{3} = 7 + 5\sqrt{2}.$$
   This finishes the proof.
\end{proof}

The collection presented in the proof of \Cref{obs:triangle_square_finite} contains a square whose sides are not parallel to the sides of $T$ having length $1$.
What if we only want to use squares whose sides are parallel to those two sides of $T$? We can prove \Cref{thm:cover_triangle_square} that states that such a collection does not exist. 

\begin{theorem} \label{thm:cover_triangle_square}
    Let $T \defeq \{(x, y)\in \mathbb{R}^2 : 0\leq x\leq 1, 0\leq  y \leq x\}$. Let $\calS$ be the countable collection of squares $S$ whose sides are parallel to the axes and $S\subset T$. If $m(\bigcup_{S\in \calS} S) \geq \frac{1-\ve}{2}$, then $\sum_{S\in \calS} \ell(S) \geq \frac{1}{2}\log \frac{1}{\ve}$. In particular, if $\bigcup_{S\in \calS} S$ contains almost all points of $T$, then the sum $\sum_{S\in \calS} \ell(S)$ diverges.
\end{theorem}
\begin{proof}
    For each $S\in \calS$, define $S^*$ be the symmetric transformation of $S$ with respect to the line $x - y = 0$. Define a graphon $G_S \defeq \mathbf{1}_{S\cup S^*}$ and consider $\calG \defeq \{G_S: S\in \calS\}$. Then  $\calG$ is a $(1-\ve)$-full separating system where each $S\in \calS$ is $(A_S, B_S)$-bipartite for some $A_S, B_S$ with $m(A_S \cup B_S) = 2\ell(S)$. By \Cref{thm:main_noweight}, we have
    $$\sum_{S\in \calS} 2\ell(S) = \sum_{S\in \calS} m(A_S\cup B_S) \geq \log\frac{1}{\ve}.$$ Thus $\sum_{S\in \mathcal{S}} \ell(S) \geq \frac{1}{2}\log\frac{1}{\ve}$, this finishes the proof.
\end{proof}

The following observation implies if we consider the sum of the $d$-th power of lengths of squares with $d>1$, then a desired collection of squares exists. 

\begin{observation}\label{obs:square-side-power}
     Let $d > 1$ and $T \defeq \{(x, y)\in \mathbb{R}^2 : 0\leq x\leq 1, 0\leq  y \leq x\}$. There exists a countable collection of squares $\calS$ whose sides are parallel to the axes and $S\subset T$ and $\sum_{S\in \calS} \ell(S)^d$ converges. 
\end{observation}
\begin{proof}
Take a square $S_1= [0,1/2]\times [0,1/2]$. Then $T-S_1$ is a disjoint union of two triangles similar to $T$ with the ratio $\frac{\sqrt{2}}{3} : 1$.
We repeat this process on each of the smaller triangles to obtain a collection $\mathcal{S}$ of squares covering almost all of $T$. Then for each $i\in \mathbb{N}$, we obtain $2^{i-1}$ triangles of side lengths $(\frac{1}{2})^i$. Moreover, as $d>1$, we have
$$\sum_{i\in \mathbb{N}}\sum_{j\in [2^{i-1}]} \ell(S_{i, j})^d = \sum_{i\in \mathbb{N}}\frac{2^{i-1}}{2^{di}} = \sum_{i\in \mathbb{N}} 2^{(1-d)i - 1} < \infty.$$ 
\end{proof}


Another potential application is regarding the following concept of perfect $k$-hashing. 
\begin{definition}
    Let $\calC$ be the collection of elements in $[k]^n$, where each element is called a $k$-ary code with length $n$. We say $\calC$ is perfect $k$-hashing with multiplicity $\lambda$ if for any size $k$ subset $S\subset \calC$, there exists $J\subseteq [n]$, $|J|\geq \lambda$ such that every $j\in J$, the $j$-th coordinates of $S$ is rainbow. If $\lambda = 1$, simply saying $C$ is perfect $k$-hashing.
\end{definition} 

 In 1984, Fredman and Koml\'{o}s~\cite{fredman1984size} obtained the upper bound of perfect $k$-hash codes using Hansel's lemma (\Cref{lem:Hansel_lemma}). They showed that if $\calC$ is an perfect $k$-hashing code, then $|\calC| \leq 2^{\left(\frac{k!}{k^{k-1}} + o(1)\right) n}$. This bound is called Fredman-Koml\'{o}s bound. Recently, Guruswami and Riazanov~\cite{guruswami2022beating} conjectured that Fredman-Koml\'{o}s bound can be improved in exponentially for any $k > 3$. They also provided explicit better bound for $k = 5, 6$ and remain other cases as a conjecture. Costa and Dalai~\cite{costa2021new} solved Guruswami and Riazanov's conjecture affirmatively, so there exist $R_k < \frac{k!}{k^{k-1}}$ for any $k > 3$ such that $|\calC|\leq 2^{R_k n}$ for any perfect $k$-hashing $\calC$.
 
 By using \Cref{cor:graph_lowerbound} instead of \Cref{lem:Hansel_lemma} in their proof, one can extend this result to conclude that $|\calC|\leq  n^{-\lceil \frac{\lambda-1}{2} \rceil + o(1)} 2^{R_k n}$ if $\calC$ is a $k$-hashing with multiplicity $\lambda$. However, as this improvement is almost negligible for small $\lambda$ compared with the original bound for $\lambda=1$, we do not claim any significance of this bound. However, this suggests a possibility that \Cref{cor:graph_lowerbound} might be useful in terms of studying $k$-hashing with multiplicity $\lambda >1$.

\section{Concluding remarks}
In this article, we investigated generalizations of Hansel's lemma for multigraphs and for graphons. Between our two results \Cref{cor:graph_lowerbound} and \Cref{thm:graph_upperbound}, there is a gap between the constants multiplied to the term $\log\log \frac{1}{\ve}$. Identifying the correct constant would be an interesting question. 

As Hansel's lemma considered weakly separating systems, there is similar inequality in strongly separating systems, which can be phrased as a digraph covering problem. Bollob\'{a}s and Scott~\cite{bollobas2007separating} show there exists similar bound exists on strongly separating systems by using Bollob\'{a}s set pairs inequality~\cite{bollobas1965generalized}. In their follow-up study~\cite{bollobas2007separating2}, they applied their inequality about strongly separating systems to show every $n$-vertex digraph with the diameter at most $2$ have arcs at least $n\log n - \frac{3}{2}\log \log n - O(n)$. It will be interesting to find generalizations of these results in terms of multi-digraph or graphons.

\section*{Acknowledgement}
The authors would like to thank Seonghyuk Im for his helpful discussions.

\printbibliography

\end{document}